\documentclass[a4paper,reqno,11pt]{amsart}

\usepackage{graphicx}
\usepackage{amsmath}
\usepackage{amsthm}
\usepackage{amssymb}
\usepackage{mathtools}
\usepackage{enumitem}
\usepackage{tikz-cd}
\usepackage[colorlinks]{hyperref}

\numberwithin{equation}{section}
\setlist[1]{leftmargin=*,label={\rm(\arabic*)}}

\theoremstyle{plain}

\newtheorem{thm}{Theorem}[section]
\newtheorem{prop}[thm]{Proposition}
\newtheorem{cor}[thm]{Corollary}
\newtheorem{lem}[thm]{Lemma}
\theoremstyle{definition}

\theoremstyle{remark}
\newtheorem{rmk}[thm]{Remark}

\newcommand{\bbN}{\mathbb{N}}
\newcommand{\bbZ}{\mathbb{Z}}

\newcommand{\emptycomment}[1]{} 

\renewcommand{\Im}{\operatorname{Im}}
\renewcommand{\top}{\operatorname{top}}
\DeclareMathOperator{\rad}{rad}
\DeclareMathOperator{\soc}{soc}

\DeclareMathOperator{\corank}{corank}
\renewcommand{\mod}{\operatorname{mod}}

\newcommand{\lto}{\longrightarrow}

\newcommand{\F}{\mathcal{F}}

\newcommand{\Q}{\mathcal{Q}}
\newcommand{\R}{\mathcal{R}}

\begin{document}
\title{The magnitude  for  Nakayama algebras}
\author{Dawei Shen, Yaru Wu}
\address{School of mathematics and statistics \\
         Henan University \\
         Kaifeng, Henan 475004 \\
         P. R. China}
\email{sdw12345@mail.ustc.edu.cn}
\address{School of mathematics and statistics \\
         Henan University \\
         Kaifeng, Henan 475004 \\
         P. R. China}
\email{wyr39@qq.com}
\date{\today}
\begin{abstract}
The magnitude for algebras is a generalization of the Euler characteristic.
We investigate the  magnitude for  Nakayama algebras.
Using Ringel's resolution quiver, the existence and the value of rational magnitude is given.
As a result, we show that the finite global dimension criteria 
for Nakayama algebras of Madsen and the first author are equivalent.
\end{abstract}
\maketitle

\section{Introduction}
Let $R$ be a commutative ring.
Given a positive integer $n$, we write $u\in R^n$
for the column vector with $u_i=1$ for
all $1\leq i\leq n$.
Let $C$ be an $n\times n$ matrix in $R$.
A weighting on $C$ is
a column vector $\alpha$
such that $C\alpha = u$.
A coweighting on $C$ is a row vector
$\beta$ such that $\beta C= u^*$, where $u^*$ is the transpose of $u$.

Following \cite{Lei12}, a matrix $C$ is said to have magnitude if it admits at least
one weighting $\alpha$ and at least one coweighting $\beta$.
The magnitude of $C$ is given by
\[m(C)=\sum\nolimits_i\alpha_i=\sum\nolimits_i\beta_i.\]

A matrix has magnitude if and only its transpose has magnitude.
If a matrix $C$ has magnitude,
then the magnitude of $C$ is independent of the weighting and the coweighting chosen.
If $C$  is an invertible matrix, then the weighting and  coweighting are unique 
and the magnitude of $C$ is the sum of all entries in the inverse matrix of $C$.

Let  $A$ be a finite dimensional algebra over a field $k$.
If $A$ has finite global dimension, 
the Cartan matrix $C_{A}$ of $A$ is invertible in the ring of integers.
The magnitude of $C_A$ is equal to the Euler characteristic of $A$.
If $A$ has infinite global dimension, the Euler characteristic is an infinite sum
while the magnitude may exist  in the field of rational numbers.

In the present paper, we study the magnitude of Cartan matrix for any Nakayama algebra.
Using Ringel's resolution quiver, we prove the existence and give the value of magnitude 
in  the field of rational numbers.
We describe syzygy of simple modules and give a special coweighting.
As a result, we show that the finite global dimension criteria 
for Nakayama algebras of Madsen and the first author are equivalent.
Also, we compute the graded Cartan determinant of any graded Nakayama algebra.

Through this paper, $k$ is a fixed field.
All modules are finite dimensional right modules.

\section{Magnitude of a Nakayama algebra}
Let $n$ be a positive integer and  $\Delta_n$ be the following quiver
\[\begin{tikzcd}
1\ar[r,"x_1"] & 2\ar[r,"x_2"]  &\cdots \ar[r,"x_{n-1}"] &
n \ar[lll,bend left=1em,"x_n"].
\end{tikzcd}\]

An algebra $A$ over a field $k$ is called a Nakayama algebra of order $n$
if $A$ is isomorphic to the quotient algebra of $k\Delta_n$ by an ideal
generated by a finite set $\R$ of nontrivial  paths.
If $\R$ consists of paths of length  $\geq 2$,
then $A$ is  a cyclic Nakayama algebra.
If there is only one arrow in $\R$,
then $A$ is  a linear Nakayama algebra.
We also allow more than  one arrow in $\R$.

Let $A$ be a Nakayama algebra of order $n$.
For any $1\leq i\leq n$, let $e_a$ be the idempotent, 
$S_a$ the simple module, $P_a$ the projective module, 
and $I_a$ the injective module at $a$.
Define $\tau(a)=a+1$ for any $1\leq a\leq n-1$ and $\tau(n)=1$.
Denote by $p_a$ the dimension of  $P_a$.
Then $p_a\leq p_{\tau(a)}+1$ for any $a$.
The sequence $(p_1,p_2,\cdots, p_n)$ is called an admissible sequence of $A$.

The resolution quiver $R_A$ is introduced by Ringel in \cite{Rin13}.
The vertices are $S_1,S_2,\cdots,S_n$  and there is
a unique arrow from  $S$ to $\gamma(S)=\tau\soc P(S)$ for each simple module $S$.
The resolution quiver of $A$ can be realized the functional graph of Gustafson function
\[\gamma(a)=\tau^{p(a)}(a)\]
for any $1\leq a\leq n$; see \cite{Gus85}.
Let $q_a$ be the dimension of  $I_a$.
The function
\[\psi(a)=\tau^{-q(a)}(a)\]
gives rise to the coresolution quiver $R^*_{A}$ of $A$.

A vertex in the resolution quiver is called cyclic if it lies on a cycle.
Given a cycle  $I$ in $R_A$,
the periodicity $p$ of $I$ is the number of vertices in $I$.
The weight $w$ of $I$ is $\sum_ap_a/n$
where $a$ runs through all vertices in $I$.
It turns out that $p$ and $w$ are relatively prime integers.
Moreover, $p$ and $w$ are independent of the cycle chosen; see \cite{She14}.

Let $\mathbf{d}$ be a grading on $A$, i.e. $\mathbf{d}$ is
a map from the set of arrows to $\bbZ$.
For any $\ell$, we write $A_\ell$ for the degree $\ell$ part of $A$.
The graded Cartan matrix $C_{A}^\mathbf{d}(t)$  is a matrix in the ring $\bbZ[t^{\pm 1}]$
of Laurent polynomials over $\bbZ$ and
\[c_{ab}(t)=\sum\nolimits_{\ell}\dim_k(e_bA_\ell e_a)t^\ell\]
for all $1\leq a,b \leq n$.
If $\mathbf{d}$ is zero, then $C_{A}^\mathbf{d}(t)$ is the Cartan matrix $C_A$.

Given  a subset $I$ of $\{1,2,\cdots, n\}$,
we write $\epsilon_I\in \mathbb{Z}^n$ for the characteristic vector of $I$.
Here, $\epsilon_{I,a}=1$ if $a\in I$ and $\epsilon_{I,a}=0$ otherwise.

\begin{lem} \label{lem:weighting}
Let  $C_A$ be the Cartan matrix of  $A$.
\begin{enumerate}
    \item If $I$ be a cycle in $R_A$, then $\epsilon_I/w$ is a weighting on $C_A$;
    \item If $\alpha$ is a weighting on $C_A$ and $a$ is not cyclic, then $\alpha_a=0$;
    \item If $\alpha$ is a weighting on $C_A$ and $a,b$ lie on the same cycle, then $\alpha_a=\alpha_b$.
\end{enumerate}
\end{lem}

\begin{proof}
(1) Denote by $\mathbf{dim}\,  M$ the dimension vector of a module $M$.
Let $a_1,a_2,\cdots,a_p$ be the vertices in $I$.
Since $I$ is a cycle in $R_A$,
then we have
\[\mathbf{dim}\,  P_{a_1}+\mathbf{dim}\,  P_{a_2}+\cdots +\mathbf{dim}\, P_{a_p}=w\cdot (1,1,\cdots,1).\]
It follows that $C_A \epsilon_I =w\cdot u$.

(2\&3) By \cite[4.1]{She17}, the number of cycles in $R_A$ is equal to $\corank C_A+1$.
Any weighting on $C_A$ is a linear combination of vector $\epsilon_I/w$, where $I$ is a cycle.
Then  the statement holds.
\end{proof}

Similarly, one can show that $C_A$ has a coweighting  in $\mathbb{Q}$.
\begin{thm} \label{thm:mag}
Let $A$ be a Nakayama algebra and $C_A$ be the Cartan matrix of $A$.
Then $C_A$ has  magnitude in  $\mathbb{Q}$.
The magnitude of $C_A$ is
\[m(C_A)=p/w\]
where $p$ is the periodicity and $w$ is the weight of $R_A$.
\end{thm}

\begin{rmk}
Since  $p,w$ are relatively prime,
it follows that
\[p(R_A)=p(R^*_{A}),\quad w(R_A)=w(R^*_{A}).\]
This is proven in \cite{She15} by different means.
\end{rmk}

Let $\mathbf{d}$ be a positive grading on $A$, i.e., $\mathbf{d}(x_i)$ is positive for  $1\leq i\leq n$.
Since $C^\mathbf{d}_{A}(0)$ is the identity matrix, 
the graded Cartan matrix $C^\mathbf{d}_{A}(t)$ is invertible 
in the field $\mathbb{Q}(t)$ of rational fractions over $\Q$.

Denote by $\alpha^\mathbf{d}(t)=(\alpha^\mathbf{d}_{a}(t))$
the unique weighting on  $C^\mathbf{d}_{A}(t)$.
Then $\alpha^\mathbf{d}_{a}(t)$ is  related to the minimal graded injective resolution of
the simple module $S_a$.
For a graded module $M$, denote by $M(a)$ the $a$-th shift of $M$.
Let
\[\xi\colon 0\to S_a \to I_{c_0}(b_0)\to I_{c_1}(b_1) \to \cdots\]
be a  minimal graded injective resolution of $S_a$.
Then we have
\begin{equation}\label{series}
\alpha^\mathbf{d}_{a}(t)=\sum\nolimits_i(-1)^i t^{b_i}.
\end{equation}

\begin{prop}
Let $A$ be a Nakayama algebra  with positive grading $\mathbf{d}$.
\begin{enumerate}
    \item If $\mathrm{id}\,S_a$ is odd, then $\alpha^\mathbf{d}_{a}(1)=0$;
    \item If $\mathrm{id}\,S_a$ is even, then $\alpha^\mathbf{d}_{a}(1)=1$;
    \item If $\mathrm{id}\,S_a$ is infinite, then $0<\alpha^\mathbf{d}_{a}(1)<1$;
    \item If $a$ is not cyclic, then $\alpha^\mathbf{d}_{a}(1)=0$;
    \item If $a,b$ lie on the same cycle, then $\alpha^\mathbf{d}_a(1)=\alpha^\mathbf{d}_b(1)$.
\end{enumerate}
\end{prop}

\begin{proof}
(1\&2) This follows directly from equation \eqref{series}.

(3) Suppose $\mathrm{id}\,S_a$ is infinite.
Since $A$ is of finite representation type, then $\xi$ is periodic after finite steps.
By \eqref{series}, there are  $f(t)$ and $g(t)$  such that
\begin{eqnarray*}
    &\alpha^\mathbf{d}_{a}(t)=f(t)+g(t)/(1-t^{b_{2r+2s}-b_{2r}}),\\
    &f(t)=t^{b_0}-x^{b_1}+\cdots+ t^{b_{2r-2}}-t^{b_{2r-1}},\\
    &g(t)=t^{b_{2r}}-t^{b_{2r+1}}+\cdots+ t^{b_{2r+2s-2}}-t^{b_{2r+2s-1}}.
\end{eqnarray*}
Since $\{b_i\}$ is strictly increasing,
then $\alpha^\mathbf{d}_{a}(1)$ is a proper fraction.

(4\&5) By (1-3) we know that $\alpha^\mathbf{d}(1)$ is a weighting on $C_A$.
Then the statements follow from Lemma \ref{lem:weighting}.
\end{proof}

Let $\mathbf{d}$ be the length grading, i.e., $\mathbf{d}({x_a})=1$ for any $1\leq a\leq n$.
Then $\alpha_a(t)=\alpha^\mathbf{d}_{a}(t)$ is determined 
by the dimension of cosyzygy  of the simple module $S_a$.
If $a$ is not cyclic, then  $\alpha_{a}(1)=0$.
It remains to calculate $\alpha_{a}(1)$ for any cyclic vertex $a$ in the resolution quiver.

\begin{lem}[{\cite[3.7]{Mad05}}]
Let  $a$ be a vertex in $R_A$.
Then $a$ is cyclic  if and only if $\mathrm{id}\,S_a$ is even or infinite.
\end{lem}

Suppose $A$ has  infinite global dimension.
Since $a$ is cyclic,
then $\mathrm{id}\,S_a$ is infinite.
Assume the minimal injective dimension $\xi$  is periodic after $2r$.
One can show that the periodicity is $2p$.
Then
\[\alpha_{a}(1)=\sum\nolimits_{0\leq i \leq p-1}\dim_k \Sigma^{2r+2i}(S_a)/nw.\]

Let $n_a$ be the cardinality of the component containing $a$ in $R_A$.
\begin{thm} \label{thm:main}
Let $A$ be a Nakayama algebra with length grading.
\begin{enumerate}
    \item If $a$ is not cyclic, then $\alpha_{a}(1)=0$;
    \item If  $a$ is cyclic, then $\alpha_{a}(1) = n_a/nw$.
\end{enumerate}
\end{thm}

Suppose  $A$ has finite global dimension.
Then we have $w=1$ and $n_a=n$ by \cite[1.1]{She17}. 
Since $a$ is cyclic, then $\mathrm{id}\,S_a$ is even.
So Theorem \ref{thm:main} holds in this case.
We study syzygy module of simple modules and prove Theorem \ref{thm:main} for
Nakayama algebras of infinite global dimension in the next section.

\begin{rmk}
In \cite{Mad05}, Madsen proves that a Nakayama algebra $A$ has finite global dimension if and only if
there is a cyclic vertex in $R_A$ with even injective dimension.
In \cite{She17}, it is shown that Nakayama algebra $A$ has finite global dimension if and only if
$R_A$ is connected of weight $1$.
Using Theorem \ref{thm:main}, one can prove the equivalence of these two criteria.
\end{rmk}

\section{Syzygy of a simple module}
Let $A$ be a Nakayama algebra.
A vertex  $a$ in  $R_A$ is called a leaf if there is no arrow ending at $a$.
By \cite[3.3]{She17}, $a$ is a leaf in $R_A$ if and only if $\mathrm{id}\,S_a=1$.
The algebra $A$ is selfinjective if and only if there is no leaf in $R_A$.

Let $e=\sum_{a\in \Im \gamma}e_a$ be the sum of the idempotent at all non-leaf vertices 
and let $\epsilon (A)=eAe$.
Then $\epsilon(A)$ is a Nakayama algebra with fewer vertices than $A$.
Moreover,  $\epsilon(A)=A$ if and only if $A$ is selfinjective.
If $A$ is a cyclic Nakayama algebra, then 
$\epsilon(A)$ coincides with Sen’s $\epsilon$-construction \cite{Sen21,Sen19}.

Let $\mod A$ be the category of finite dimensional modules over $A$.
By \cite[3.6]{HI20}, there is an adjoint pair of exact functors
\[\begin{tikzcd}
    \mod A \ar[r,shift right=0.3em, "F=(-)e" below] &\mod\epsilon(A)
    \ar[l,shift right=0.3em,"G=(-)\otimes_{eAe} eA" above].
\end{tikzcd}\]

Denote by $\F$  the image of the functor $GF$
and by $E(S_a)$ the image  of $S_a$ for each $a$ in $R_A$.
Then $\{E(S_a)\mid a\in \Im\gamma\}$
is a complete set of simple objects in $\F$.
Moreover, $\F$ is an exact subcategory of $\mod A$ and $\F$ is closed under projective covers.
\begin{enumerate}
    \item  If $a\notin \Im \gamma$, then $E(S_a)=0$;
    \item  If $a\in \Im \gamma$,  then $\top E(S_a)=S_a$;
    \item  $\sum_{a}\mathbf{dim}\,E(S_a)=(1,1,\cdots,1)$;
    \item  $\Omega^{2}(M)\in \F$ for any $M$ in $\mod A$.
\end{enumerate}

\begin{lem} \label{lem:fis}
Let $M$ be an indecomposable module.
\begin{enumerate}
    \item\cite[3.2]{Rin13} Either $\mathrm{pd}\,M\leq 1$ or else $\top \Omega^{2}(M)=\gamma(\top M)$.
    \item\cite[A.1]{Rin21} The map $\gamma\colon \Im\psi\to \Im\gamma$ is bijective with inverse $\psi$.
\end{enumerate}
\end{lem}

\begin{lem}\cite[3.1]{Mad05}
Let $S_a$ and $S_b$ be simple modules with $S_a$  not being projective.
Then $S_b$ is a subfactor  of $\Omega^2(S_a)$ (with multiplicity $1$) if and only if $\psi(b)=a$.
\end{lem}

\begin{prop} \label{prop:ele}
Let $S_a$ and $S_b$ be simple modules.
\begin{enumerate}
   \item If $a\in \Im \gamma$, then $S_b$ is a subfactor of $E(S_a)$ if and only if $\psi(b)=\psi(a)$.
   \item If $\mathrm{pd}\,S_a\geq 2$, then  $\Omega^2(S_a)\cong E(S_{\gamma(a)})$.
\end{enumerate}
\end{prop}

\begin{proof}
(1) Let $S_b\neq S_a$ be a subfactor of $E(S_a)$, then  $\mathrm{id}\,S_b=1$.
The sequence
\[0\to S_b\to I_b \to I_{\tau^{-1}b}\to 0\]
is exact. Then $\psi(b)=\psi(\tau^{-1}b)$.
If $\tau^{-1}(b)=a$, we are done.
If $\tau^{-1}(b)\neq a$, then we can repeat this process.
It follows that  $\psi(b)=\psi(a)$. Observe that
\[\sum_{a\in \Im \gamma}\dim_kE(S_a)\leq \sum_{a \in \Im\gamma}|\psi^{-1}(\psi(a))|
=\sum_{a' \in \Im\psi}|\psi^{-1}(a')|.\]
The last equality follows by Lemma \ref{lem:fis}(2).
Then the conclusion holds.

(2) By Lemma \ref{lem:fis}(1) the module $\Omega^2(S_a)$ and $E(S_{\gamma(a)})$ have the same top.
Since $\mathrm{pd}\, S_a\geq 2$, then $a=\psi\gamma(a)$.
A simple module $S_b$ is a subfactor of
$\Omega^2(S_a)$ if and only if $\psi(b)=\psi\gamma(a)$.
Then $\Omega^2(S_a)$ and $E(S_{\gamma(a)})$ have the same composition series.
They are isomorphic.
\end{proof}
Repeatedly removing the leaves, we obtain a sequence
\[\begin{tikzcd}
\mod A \ar[r,shift right=0.3em]
&\mod\epsilon(A)\ar[l,shift right=0.3em]\ar[r,shift right=0.3em]
&\cdots \ar[l,shift right=0.3em] \ar[r,shift right=0.3em]
&\mod\epsilon^r(A)\ar[l,shift right=0.3em]
\end{tikzcd}\]
of adjoint pair of exact functors
such that $\epsilon^r(A)$ is selfinjective.

Let $\F^m$ be the image of $\mod \epsilon^m(A)$ in $\mod A$,
and $E^m(S_a)$ be the image of $S_a$  for any $1\leq m\leq r$.
Then $\{E^m(S_a)\mid a\in \Im\gamma^m\}$ is a complete set of simple objects in $\F^m$.
Moreover, $\F^m$ is an exact subcategory of $\mod A$ and  is closed under projective covers.
\begin{enumerate}
    \item  If $a\notin \Im \gamma^m$, then $E^m(S_a)=0$;
    \item  If $a\in \Im \gamma^m$,  then $\top E^m(S_a)=S_a$;
    \item  $\sum_{a}\mathbf{dim}\,E^m(S_a)=(1,1,\cdots,1)$;
    \item  $\Omega^{2m}(M)\in \F^m$ for any $M$ in $\mod A$.
\end{enumerate}

\begin{lem} \label{lem:gamma-psi}
Let $m\in \bbN$ and $M$ be an indecomposable  module.
\begin{enumerate}
    \item \cite[3.3]{Rin13} Either $\mathrm{pd}\,M\leq 2m-1$ or else $\top \Omega^{2m}(M)=\gamma^m(\top M)$.
    \item \cite[A.1]{Rin21} The map $\gamma^m\colon \Im\psi^m\to \Im\gamma^m$ is bijective with inverse $\psi^m$.
\end{enumerate}
\end{lem}

\begin{prop} \label{prop:mth}
Let $m\in \bbN$ and $S_a,S_b$ be simple modules.
\begin{enumerate}
   \item If $a\in \Im \gamma^m$, then $S_b$ is a subfactor of $E^m(S_a)$ if and only if $\psi^m(b)=\psi^m(a)$.
   \item If $\mathrm{pd}\,S_a\geq 2m$, then  $\Omega^{2m}(S_a)\cong E^m(S_{\gamma(a)})$.
\end{enumerate}
\end{prop}

\begin{proof}
(1) Let $S_b$ be a subfactor of $E^m(S_a)$.
If $m=1$,  then $\psi(b)=\psi(a)$ by Proposition \ref{prop:ele}(1).
If $m\geq 2$, then $S_b$ is a subfactor of an $E^{m-1}(S_{a'})$ with $\psi(a')=\psi(a)$.
By induction hypothesis,  $\psi^{m-1}(b)=\psi^{m-1}(a')$.
Then
\[\psi^m(b)=\psi\psi^{m-1}(b)=\psi\psi^{m-1}(a')=\psi\psi^{m-1}(a')=\psi^m(a).\]
By  Lemma \ref{lem:gamma-psi} we have
\[\sum_{a\in \Im \gamma^m}\dim_kE^m(S_a)\leq \sum_{a \in \Im\gamma^m}|\psi^{-m}(\psi^m(a))|
=\sum_{a' \in \Im\psi^m}|\psi^{-m}(a')|.\]
Then the equality holds.

(2) By Lemma \ref{lem:gamma-psi}(1) the top of $\Omega^{2m}(S_{a})$ and $E^m(S_{\gamma^m(a)})$ are equal.
For any $k\in \bbN$, the operation $\Omega^{2}$ sends simple objects in $\F^{k-1}$ to simple objects in $\F^{k}$ or zero.
Then  $\Omega^{2m}(S_a)$ is isomorphic to  $E^m(S_{\gamma^m(a)})$
\end{proof}

Let $m\in \bbN$.
Following \cite[3.5]{Rin21} we get that $\mathrm{pd}\,S_a\neq 1,3,\cdots, 2m-1$  if and only if $a \in \Im\psi^m$ .

\begin{prop} \label{prop:iff}
Let $m\in \bbN$ and $S_a$ be a simple module.
If $\mathrm{pd}\, S_{a}\geq 2m-1$,
then $S_b$ is a subfactor of $\Omega^{2m}(S_{a})$ if and only if $\psi^m(b)=a$.
\end{prop}

\begin{proof}
If  $\mathrm{pd}\, S_{a}= 2m-1$, then  $a\notin \Im\psi^m$ and $\Omega^{2m}(S_{a})$ is zero.
If  $\mathrm{pd}\, S_{a}\geq 2m$, then $a\in \Im\psi^m$ and  $\psi^m\gamma^m(a)=a$.
Since $\Omega^{2m}(S_a)$
and $E^m(S_{\gamma^m(a)})$ are isomorphic,
$S_b$ is a subfactor of $\Omega^{2m}(S_{a})$
if and only if $\psi^m(b)=a$.
\end{proof}

\begin{thm}
Let $A$ be a Nakayama algebra, and $S_a$ be a simple module.
Suppose a minimal projective resolution of $S_a$ is periodic after $2r$. Then
\[\sum\nolimits_{0\leq i\leq p-1}\dim_k\Omega^{2r+2i}(S_{a})=n'_{a}\]
where $p$  is the periodicity of  $R^*_{A}$  and $n'_{a}$ is the cardinality
of the component containing $a$  in $R^*_{A}$.
\end{thm}

Note that Theorem \ref{thm:main} follows by a dual version of this theorem.

\begin{proof}
By Proposition \ref{prop:iff} a simple module $S_b$ is a subfactor of $\Omega^{2r+2i}(S_a)$
if and only if $\psi^{r+i}(b)=a$.
Then
\[\sum\nolimits_{0\leq i\leq p-1}\dim_k\Omega^{2r+2i}(S_{a})=
\sum\nolimits_{0\leq i \leq p-1}|\psi^{-r-i}(a)|=n'_a.\]
The proof is finished.
\end{proof}

\section{Graded Cartan determinant}
We calculate the determinant of the graded Cartan matrix of a Nakayama algebra with grading.

Let $A$ be a Nakayama algebra.
Suppose $A$ is not selfinjective.
Then the resolution quiver $R_A$ has a leaf.
Let $n$ be a leaf and let $e'=1-e_n$.
The algebra $A'=e'Ae'$  is a Nakayama algebra of order $n-1$.

For any grading  $\mathbf{d}$ is on $A$,
let $\mathbf{d'}$ be the grading on $A'$ where
\[\mathbf{d'}(x_a)=\begin{cases} \mathbf{d}(x_{a})+\mathbf{d}(x_{a+1})& \text{if}\; a=n-1 \\
    \mathbf{d}(x_{a}) & \text{otherwise}.
\end{cases}\]
\begin{lem}[{\cite[3.8]{HI20}}] \label{lem:res}
Suppose $n$ is a leaf in $R_A$.
\begin{enumerate}
    \item  $R_{A'}$ equals  $R_A$ with $n$ removed;
    \item  $R_{A'}$ and $R_A$ have the same weight;
    \item  $R_{A'}$ and $R_A$ have the same number of component.
\end{enumerate}
\end{lem}

\begin{lem} \label{lem:det}
We have $\det C^\mathbf{d}_{A}(t)=\det C^\mathbf{d'}_{A'}(t)$.
\end{lem}

\begin{proof}
Since $n$ is a leaf, the injective dimension of $S_n$ is equal to $1$.
There is an exact sequence of graded modules
\[0\to S_n \to I_n \overset{\alpha_{n-1}}\lto I_{n-1}(\mathbf{d}({x_{n-1}}))\to 0.\]

Denote by $r_a$ the $a$-th row of $C^\mathbf{d}_{A}(t)$.
It is the graded dimension vector of the injective module $I_a$.
The above sequence shows that
\[r_n-t^{\mathbf{d}({x_{n-1}})}r_{n-1}=\epsilon_n.\]

Adding the product of $-t^{\mathbf{d}({x_{n-1}})}$ times of $r_{n-1}$ to $r_n$
yields an elementary transformation
\[C^\mathbf{d}_{A}(t) \to \begin{pmatrix}C_1(t) & \ast \\ 0 & 1\end{pmatrix}.\]
We have $\det C^\mathbf{d}_{A}(t)=\det C_1(t)$.
It is routine to check that $C_1(t)$ is just the graded Cartan matrix of  $A'$ with grading $\mathbf{d'}$.
\end{proof}

Let $R_A$ be the resolution quiver of $A$.
Denote by $c$ the number of components in $R_A$
and by $w$ the weight of $R_A$.

\begin{lem} \label{lem:self}
Let $A=k\Delta_n/\rad^\ell$ be a selfinjective Nakayama algebra.
Let $\mathbf{d}$ be a grading on $A$ and  $d$ be the sum of degree of all arrows.
Then
\begin{enumerate}
    \item $\det C^\mathbf{d}_{A}(t)=(1-t^{d\ell/(n,\ell)})^{(n,\ell)}/(1-t^d)$;
    \item $w=\ell/(n,\ell)$ and $c=(n,\ell)$.
\end{enumerate}
\end{lem}

\begin{proof}
(1) Let $D$ be the $n\times n$ matrix
\[D=\begin{pmatrix}
0 & 0 & \cdots &0 & t^{\mathbf{d}({\alpha_n})}\\
t^{\mathbf{d}({\alpha_1})} & 0 & \cdots & 0&0\\
0 & t^{\mathbf{d}({\alpha_2})} & \cdots & 0&0\\
\vdots & \vdots & \ddots  & \vdots & \vdots\\
0 & 0 & \cdots & t^{\mathbf{d}({\alpha_{n-1}})} &0\\
\end{pmatrix}\]
and let $\mathbb{I}$ be the $n\times n$ identity matrix.
Then
\[C^\mathbf{d}_{A}(t)=\mathbb{I}+D+\cdots+D^{\ell-1}=\frac{\mathbb{I}-D^{\ell}}{\mathbb{I}-D}.\]
We have $\det(\mathbb{I}-D^k)=(1-t^{dm/(n,k)})^{(n,k)}$ for any $k\geq 1$.
It follows that
\[\det C^\mathbf{d}_{A}(t)=(1-t^{d\ell/(n,\ell)})^{(n,\ell)}/(1-t^d).\]

(2) The admissible sequence of $A$ is $(\ell,\ell,\cdots, \ell)$.
Then $R_A$ consists of $(n,\ell)$ cycles of size $n/(n,\ell)$.
We have $w=\ell(n,\ell)$ and $c=(n,\ell)$.
\end{proof}

\begin{thm}
Let $A$ be a Nakayama algebra with grading $\mathbf{d}$.
Then the determinant of the graded Cartan matrix $C^\mathbf{d}_{A}(t)$ is
\[\det C^\mathbf{d}_{A}(t)=(1-t^{d w})^{c}/(1-t^d)\]
where $d$ is the sum of degree of all the arrows.
\end{thm}

\begin{proof}
If $A$ is selfinjective, then the conclusion follows from Lemma \ref{lem:self}.

Suppose that $A$ is not selfinjective.
There is a  Nakayama algebra $A'$ with grading $\mathbf{d'}$.
By Lemma \ref{lem:det} the determinants
of $C^\mathbf{d}_{A}$ and $C^{\mathbf{d}'}_{A'}$ are equal.

By Lemma \ref{lem:res} $R_A$ and $R_{A'}$ have the same weight and the number of cycles.
Note that $\mathbf{d}'$ and $\mathbf{d}$ have the same total degree.
Repeat this process, we get a selfinjective Nakayama algebra $B$ with grading $\mathbf{f}$.
Then
\[\det C^\mathbf{d}_{A}(t)=\det C^\mathbf{f}_{B}(t)=(1-t^{d\cdot w})^{c}/(1-t^d).\]
This finishes the proof.
\end{proof}

\begin{cor}
 Let $A$ be a Nakayama algebra of order $n$ with length grading.
 Then the determinant of the graded Cartan matrix $C_{A}(t)$ is
\[\det C_{A}(t)=(1-t^{nw})^{c}/(1-t^n)\]
\end{cor}

\section*{Acknowledgements}
The authors are supported by the National Natural Science Foundation of
China No. 11801141.
\

\end{document}